\documentclass{article}
\usepackage{amsmath}
\usepackage{amssymb}
\usepackage{xspace}
\usepackage{amsthm}
\usepackage{epsfig,graphicx}
\usepackage {color, tikz}

\newcommand{\C}{\mathbb C}
\newcommand{\R}{\mathbb R}
\newcommand{\Z}{\mathbb Z}

\renewcommand{\P}{\mathbb{P}}
\newcommand{\Xd}{X}
\newcommand{\Li}{\mathcal{L}}
\newcommand{\LL}{\lvert \Li \rvert}
\newcommand{\D}{\mathcal{D}}
\newcommand{\Log}{\operatorname{Log}}
\newcommand{\ttor}{( \C^\ast )^2}
\newcommand{\A}{\mathcal{A}}

\newcommand\cS{{\ensuremath{\mathcal{S}}}\xspace}
\newcommand\cV{{\ensuremath{\mathcal{V}}}\xspace}
\newcommand{\ly}{\mathcal{LL}}
\newcommand{\TG}{ \tilde{\gamma}}
\newcommand{\cp}[1]{\C P^{#1}}
\newcommand{\rp}[1]{\R P^{#1}}

\newtheorem{Proposition}{Proposition}[section]
\newtheorem{Definition}[Proposition]{Definition}
\newtheorem{Lemma}[Proposition]{Lemma}
\newtheorem{Remark}[Proposition]{Remark}
\newtheorem{Theorem}{Theorem}
\newtheorem{Conjecture}{Conjecture}
\newtheorem{Corollary}[Proposition]{Corollary}
\newtheorem{Question}{Question}

\newtheorem*{ack}{Acknowledgement}

\DeclareMathOperator{\itr}{int}
\DeclareMathOperator{\conv}{conv}

\begin{document}

\title{Amoebas of curves and the Lyashko-Looijenga map}
\author{Lionel Lang}

\maketitle

\begin{abstract}
For any curve $\cV$ in a toric surface $\Xd$, we study the critical locus $S(\cV)$ of the moment map $\mu$ from $\cV$ to its compactified amoeba $\mu(\cV)$. We show that for curves $\cV$ in a fixed complete linear system, the critical locus $S(\cV)$ is smooth apart from some real codimension $1$ walls. We then investigate the topological classification of pairs $(\cV,S(\cV))$ when $\cV$ and $S(\cV)$ are smooth. As a main tool, we use the Lyashko-Looijenga mapping ($\ly$) relative to the logarithmic Gauss map $\gamma : \cV \rightarrow \cp{1}$. We prove two statements concerning $\ly$ that are crucial for our study: the map $\ly$ is algebraic; the map $\ly$ extends to nodal curves. It allows us to construct many examples of pairs $(\cV,S(\cV))$ by perturbing nodal curves.
\end{abstract}

\footnote{{\bf MSC-2010 Classification:} 14H50, 14M25}

\section{Introduction}

Establishing a bridge between algebraic and tropical geometry, amoebas have generated a lot of interest since their introduction in \cite{GKZ}, see $e.g.$ \cite{FPT},  \cite{Mikh}, \cite{EKL}, \cite{KO}, \cite{Kri}, \cite{MikRen}. In this paper, we focus on the case of planar curves and study a topological problem related  to their amoeba. 

For a curve $\cV$ in a complete toric surface $\Xd$, its compactified amoeba is the image $\mu(\cV)$ of $\cV$ itself by the moment map $\mu : \Xd \rightarrow \Delta$ to the moment polygon $\Delta\subset \R^2$ of $\Xd$.  Consider the critical locus $S(\cV) \subset \cV$ of the map $\mu$, $i.e.$ the set of point for which $\mu_{\vert_\cV}$ is not a local diffeomorphism. We study the following.

\begin{Question}\label{qu1}
Given an ample line bundle $\Li$ on $\Xd$, what are the possible topological pairs $(\cV,S(\cV))$ for the curves $\cV$ in the linear system $\LL$?
\end{Question}

In \cite{Mikh}, Mikhalkin gives a description of $S(\cV)$ for smooth curves $\cV$ in terms of the logarithmic Gau\ss{} map $\gamma : \cV \rightarrow \cp{1}$, namely
\begin{eqnarray}\label{eq:s}
 S(\cV) & = & \gamma^{-1}(\rp{1}).
 \label{eq:critloc}
\end{eqnarray}
It constrains the topology of $S(\cV)$ to a union of smooth ovals in $\cV$ intersecting at worst transversally, see Proposition \ref{prop:S(f)}. In the same work, Mikhalkin also noticed that, whenever $\cV$ is defined over $\R$, the real part $\R \cV \subset \R \Xd$ is contained in $S(\cV)$. It directly relates Question \ref{qu1} to \textit{Hilbert's sixteenth problem} on the topology of real algebraic curves in the plane. 

Define the discriminantal set $\cS \subset \LL$ to be the set of curve $C$ such that $S(\cV)$ is not smooth. We prove the following.

\begin{Theorem}
The set $\cS$ is a semi-algebraic set of codimension 1. Moreover, $\cS$ is the closure of its maximal stratum.
\end{Theorem}

\noindent
As a consequence, $S(\cV)$ is generically a disjoint union of smooth ovals. A simpler but still challenging variation of Question \ref{qu1} is the following.

\begin{Question}\label{qu2}
Given an ample line bundle $\Li$ on $\Xd$, what are the possible topological pairs $(\cV,S(\cV))$ for smooth curves $\cV \in \LL \setminus \cS$ ? In particular, what are the possible values for the number $b_0(S(\cV))$ of connected components of $S(\cV)$?
\end{Question}

In the context of Question \ref{qu2}, the description \eqref{eq:s} implies the following constraints 
\begin{eqnarray}
	1 \leq b_0 \big( S(C) \big) \leq deg(\gamma).
\label{eq:trivbd0}
\end{eqnarray}
Generically, the latter upper bound does not depend on $\cV$, see Proposition \ref{deglogGauss}. We have the following.

\begin{Theorem}
Let $\Xd=\cp{2}$ and $\Li = \mathcal{O}_\Xd(d)$ for any integer $d\geq1$. Then, $b_0(S(\cV))$ can achieve all the values between $d(d+1)/2$ and $d^2$. In particular, the upper bound of \eqref{eq:trivbd0} is sharp.
\end{Theorem}

In the above theorem, we are far from the trivial lower bound given in \eqref{eq:trivbd0}. In Section \ref{sec:top}, we motivate Conjecture \ref{conj} given in Section \ref{sec:mr} stating that the sharp lower bound for $b_0(S(\cV))$ is given by $g+1$ where $g$ is the genus of a generic curve in $\LL$. Moreover, this bound is always realized by simple Harnack curves, see \cite{Mikh}. 

The logarithmic Gau\ss{} map obviously plays a central role in Questions \ref{qu1} and \ref{qu2}. In particular, the description \eqref{eq:s} indicates that the topology of $S(\cV)$ depends strongly on the relative position of the critical values of $\gamma$ on $\cp{1}$ with respect to $\rp{1}$. 

The \textit{ad hoc} tool here is the Lyashko-Looijenga mapping. It is a fundamental tool in singularity theory and the study of Hurwitz spaces (see $e.g.$ \cite{Vas}, \cite{ELSV} and references therein). In the present context, it appears as a map
\[ \ly : \LL \dashrightarrow \cp{m} \]
defined on an open Zariski subset of the source: it maps a curve $\cV$ to the unordered set of critical values of  $\gamma$. Equivalently, it maps $\cV$ to the branching divisor of $\gamma$ encoded as a point in $\cp{m}$, provided that $m$ is the degree of the latter divisor.

The first result we prove concerning this map is the key for Theorem \ref{thm:sa}.

\begin{Theorem}
The map $\ly : \LL  \dashrightarrow \cp{m}$ is algebraic.
\end{Theorem}

In order to construct various topological type for the pairs $(\cV,S(\cV))$, the set of singular curves plays an important role, as it does in \textit{Hilbert's sixteenth problem}. It allows to build $S(\cV)$ inductively on the geometric genus of $\cV$. Starting from a curve with nodal points, we aim to understand the mutations of $S(\cV)$ after smoothing the nodal points one by one. As the main ingredient in the proof of Theorem \ref{thm:extLL}, we prove the following.

\begin{Theorem}
The map $\ly$ extends algebraically to the set of nodal curves in $\LL$. Moreover, we provide the explicit formula \eqref{eq:extLL} for the extension.
\end{Theorem}

Although it is beyond the present interest, we can wonder about the possibility to extend $\ly$ to the whole $\LL$ and further applications. Already, it is interesting to notice the similarities between the formula \eqref{eq:extLL} and the one given in \cite[Section 4.2]{ELSV}.

In the very last section of this paper, we illustrate the present results by explicit computation of examples.\\

Before concluding the introduction, we wish to discuss briefly some applications and connections with different topics.

Amoebas has been already studied by several people (\cite{FPT}, \cite{Ru}, \cite{Mikh}, \cite{dWS},...), generally via their contour, $i.e.$ the set of critical values of the map $\mu$. As the contour of amoebas can have singularities that survive to small perturbations, the set $S(\cV)$ is a more reasonable object to study and has not been investigated much so far.

For a real algebraic curve $\cV \in \LL$, the real locus $\R \cV$ is contained in $S(\cV)$. If the inclusion $\R \cV \subset \cV$ solely does not give much insight on the topological pair $(\R \Xd, \R \cV)$, the sequence of inclusions $\R \cV \subset S(\cV) \subset \cV$ is much more valuable as it gives informations on the amoeba $\mu(\R \cV)\subset \R^2$.

To insist on the importance of the critical locus, recall that simple Harnack curves and their generalization \cite{L2} are characterized by the containment $\R \cV \subset S(\cV)$ being an equality.

Finally in \cite{It}, Itenberg provides some obstructions to construct certain topological type of real algebraic curves in the plane via patchworking methods. A better knowledge of the critical locus would certainly lead to a better understanding of the limitation of tropical methods in constructing real algebraic curves with prescribed topology.

\begin{ack}
The author is grateful to Benoit Bertrand, R\'{e}mi Cr\'{e}tois, Ilia Itenberg, Nikita Kalinin, Grigory Mikhalkin, Johannes Rau and especially to Timo de Wolff for useful discussions and comments. 
This work has been carried over several years during which the author has been supported by the grant ERC TROPGEO, the FNS project 140666 and the Knut and Alice Wallenberg Foundation.
\end{ack}

\newpage
\tableofcontents

\setcounter{Theorem}{0}

\section{Setting and statements}\label{sec:stat}

\subsection{Curves in toric surfaces}

A Laurent polynomial $f \in \C \left[z^{\pm 1}, w^{\pm 1}\right]$ is supported on a finite subset $A \subset \Z^2$ if it can be written
\[ f(z,w)= \sum_{(a,b) \in A} c_{(a,b)} z^aw^b\]
with $c_{(a,b)} \in \C^\ast$. We denote by $\cV(f) := \left\lbrace (z,w) \in \ttor \, \vert \, f(z,w)=0 \right\rbrace$ the curve given by any such polynomial $f$. Define the amoeba $\A(f)$ of the curve $\cV(f)$ to be the image of $\cV(f)$ itself by the amoeba map
\[
\begin{array}{rcl}
\A  : \ttor & \rightarrow & \mathbb{R}^2 \\
(z,w) & \mapsto & \big( \log \vert z \vert, \log \vert w \vert \big)
\end{array}
\]
Whenever $\cV(f)$ is smooth, we can define the critical locus of the restriction of $\A$ to $\cV(f)$ by
\begin{equation}\label{eq:critloc}
S(f) := \left\lbrace p \in \cV(f) \; \vert \;  T_p \mathcal{A} : T_p \cV(f) \rightarrow \R^2 \text{ is not submersive} \right\rbrace.
\end{equation}

We now compactify  of the above picture. Let $\Delta \subset \R^2$ be a bounded convex lattice polygon of dimension 2. The polygon $\Delta$ defines a smooth complete toric surface $\Xd_\Delta$ as follows: label the elements of $\Delta \cap \Z^2 = \{ (a_0,b_0),...(a_m,b_m)\}$ and consider the monomial embedding  $(\C^\ast)^2 \rightarrow \cp{m}$ given by
\[ (z,w) \mapsto \left[ z^{a_0}w^{b_0}:...: z^{a_m}w^{b_m}\right]. \]
Then $\Xd_\Delta$ is defined as the closure of $(\C^\ast)^2$ in $\cp{m}$. 
Denote by $\Li_\Delta$ the line bundle on $\Xd_\Delta$ given by the inclusion in $\cp{m}$ and by $\vert \Li_\Delta \vert := \P (H^0(\Xd_\Delta,\Li_\Delta))$ the space of hyperplane sections of $\Xd_\Delta \subset \cp{m}$. In the rest of the paper, we will often use the notation $\Xd$ and $\Li$ for $\Xd_\Delta$ and $\Li_\Delta$ respectively. Then, any hyperplane section $s \in \LL$ is given as the closure $\overline{\cV(f)} \subset \Xd$ for some Laurent polynomials $f$ whose support is contained in $\Delta \cap \Z^2$. Then, the elements of $\LL$ corresponds to projective classes $[f]$, or equivalently to $\cV(f)$ for such polynomials $f$. To avoid cumbersome notation,  we do not distinguish the curve $\cV(f)$ from its closure in  $\Xd$.

Up to a reparametrization identifying $\R^2$ with $\itr(\Delta)$, the map $\A$ is compactified by the moment map $\mu : \Xd_\Delta \rightarrow \Delta$, see \cite{GKZ}.

For a generic $\left[f\right] \in \LL$, the curve $\cV(f)$ has genus $g_\Delta := \itr(\Delta) \cap \Z^2$, see \cite{Kho}. If we denote the union of toric divisors $\Xd^\infty:= \Xd \setminus \ttor$, then we have the following intersection multiplicity  
$$\cV(f) \cdot \Xd^\infty = b_\Delta$$ 
where $b_\Delta := \partial \Delta \cap \Z^2.$ Defining $\cV(f)^\infty := \cV(f) \cap \Xd^\infty$, it follows that $\cV(f)^\infty$ consists in at most $b_\Delta$ many points. The cardinal of $\cV(f)^\infty$ is maximal whenever the restriction of $f$ to any edge of $\Delta$, which amounts to a univariate polynomial, has distinct roots. Hence, $\vert  \cV(f)^\infty \vert = b_\Delta$ for generic $[f] \in \LL$.

Following \cite{GKZ},  the discriminant associated to $\Delta$ refers to the space of singular curves inside $\LL$. If moreover $\vert \Delta \cap Z^2 \vert \geq 4$, the discriminant is always given as an algebraic hypersurface, see \cite[Ch.1, Sec.1]{GKZ}. For any $k\geq 0$, let us denote by $\D_k \subset \LL$ the set of curves having exactly $k$ double points and no other singularities. Then, $\D_0$ corresponds to the set of smooth curves, $\D_k$ is an open subset of $\overline{\D_{k-1}}$ for $1\leq k \leq g_\Delta$ and $\D_k$ is empty for $k>g_\Delta$. Finally, define $$ \D := \cup_{1\leq l \leq g_\Delta} \D_l.$$

\subsection{The logarithmic Gau\ss{} map}

 For a smooth curve $\cV(f) \subset \Xd$, the logarithmic Gau\ss{} map $\gamma_f  :  \cV(f) \rightarrow \cp{1}$ is given on $\ttor$ by
\[  \gamma_f (z,w) = \Big[ z \cdot \partial_z f(z,w) : w \cdot \partial_w f(z,w) \Big]\]
This map is locally the composition of (any branch of) the coordinatewise complex 
logarithm with the classical Gau\ss{} map. 

In the singular case, $\gamma_f$ is only defined on the smooth part of $\cV(f)$. We solve this issue as follows. Denote by $ \pi :  \tilde{\cV}(f) \rightarrow \cV(f)$ the normalization of $\cV(f)$. It is well defined up to automorphism of the source, but this ambiguity will not matter for our purpose. When $\cV(f)$ is smooth, $\pi$ is simply the identity. Riemann's removable singularity Theorem \cite{GR} allows to extends the rational map $\pi \circ \gamma_f$ to an algebraic map
\begin{equation}\label{eq:tg}
\TG_f  :  \tilde{\cV}(f) \rightarrow \cp{1}.
\end{equation}

For any topological space $S$, denote by $\chi(S)$ its Euler characteristic. Define also  $\tilde{\cV}(f)^\circ := \tilde{\cV}(f) \setminus \Xd^\infty$.

\begin{Proposition}[\cite{L2}]
\label{deglogGauss}
Assume that $ \pi : \tilde{\cV}(f) \rightarrow \cV(f)$ is an immersion. Then we have 
\[ deg (\TG_f) =  - \chi \big( \tilde{\cV}(f) \big) + \vert \cV(f)^\infty \vert. \]
In particular, $deg (\TG_f) = vol(\Delta)$ for a generic $[f] \in \LL$,.
\end{Proposition}
\noindent
Here, $vol$ is twice the Euclidean area. 

\begin{Remark}\label{rmk:degree}
In the above statement, we mean by generic that $\cV(f)$ is smooth and intersect $\Xd^\infty$ transversally. In general, a tangency point of order $k$ between $\cV(f)$ and $\Xd^\infty$ is responsible for the decreasing of $deg(\TG_f)$ by $(k-1)$ compared to the transversal case.
\end{Remark}

\begin{Corollary}\label{cor:binom} 
Let $f$ be a Laurent polynomial such that $\cV(f)$ is smooth, reduced and such that $\gamma_f$ is constant. Then up to multiplication by a monomial, $f =z^aw^b-c$ where $c \in \C^\ast$ and $a$ and $b$ are coprime.
\end{Corollary}

If $f$ is as in the above corollary,  then $\cV(f) \cap \ttor \simeq \C^\ast$ is a cylinder and $\gamma_f (\cV(f)) = \left\lbrace \left[a:b\right] \right\rbrace$. For these reasons, we refer to $\cV(f)$ as a binomial cylinder.

\subsection{The Lyashko-Looijenga mapping}

Let us first fix some notations concerning divisors on Riemann surfaces. A divisor $D$ on a Riemann surface $S$ is a formal sum of point
$$ D := \sum_{p\in S} D(p) \cdot p$$
where $D(p) \in \Z$ is zero for all but finitely many points $ p\in S$. Denote the degree of $D$ by $deg(D):=\sum_{p\in S} D(p)$. Recall that for any
holomorphic map $ F : S \rightarrow \cp{1}$, the ramification divisor $R_F$ on $S$ is defined by
\[R_F := \sum_{p\in S} ord_F(p) \cdot p \]
where the order $ord_F(p)$ of $F$ at $p$ is the smallest integer $n$ so that $F^{(n+1)}(p) \neq 0$. The branching divisor $B_F$ is the push-forward $F_\ast (R_F)$ on $\cp{1}$. Explicitly, 
\[B_F := \sum_{v \in \cp{1}} \Big(\sum_{p\in F^{-1}(v)} ord_F(p) \Big) \cdot v \]
The divisors $R_F$ and $B_F$ have the same degree given explicitly by the Riemann-Hurwitz formula 
\begin{eqnarray}
\label{eq:RHformula}
\chi \left( S \right) = deg( f ) \cdot \chi \left( \cp{1} \right) - deg( R_f ),
\end{eqnarray}
The set of divisors of degree $m$ on $\cp{1}$ is isomorphic to the set $\cp{m}$ of univariate polynomials of degree $m$: to every such polynomial, this isomorphism associates the formal sum of its roots counted with multiplicities, see \cite{GKZ}.  Note also that $ \cp{m}$ is isomorphic to the $m$-th symmetric product $Sym_m(\cp{1})$, see \cite[Ch.4, Sec.2]{GKZ}.

The Lyashko-Looijenga mapping associates to any holomorphic function \linebreak $F : S \rightarrow \cp{1}$ its branching divisor $B_F$, see \cite{ELSV}. If $B_F$ has degree $m$, we encode the latter as a point in $\cp{m}$. We now transpose this construction to the present context. For a generic $\left[f\right] \in \LL$, the curve $\cV(f)$ is smooth and the degree of $\gamma_f$ is equal to $vol(A)$, see Proposition \ref{deglogGauss}. The Riemann-Hurwitz formula implies then that the divisor $B_{\gamma_f}$ has constant degree for such $f$. Let $m$ be this degree. It follows that we can define the Lyashko-Looijenga mapping
\[
\begin{array}{rcl}
\ly : \LL & \dashrightarrow & \cp{m}\\
\left[f\right] & \mapsto & B_{\gamma_f}
\end{array}
\]
associated to the Logarithmic Gau\ss{} map on an open dense subset of $\LL$.

\subsection{Main results}\label{sec:mr}

Let us extend the definition of the critical locus \eqref{eq:critloc} to any curve $\cV(f) \in \LL$ by 
\begin{equation}\label{eq:sf}
S(f) := \TG_f^{-1} (\rp{1}) \subset \tilde{\cV}(f).
\end{equation}
and define
\begin{equation}\label{eq:SA}
	\cS_\Delta  =  \{ \left[f\right] \in \LL \; \vert \:  S(f) \text{ is singular}\} 
\end{equation}
Note that $S(f)$ corresponds to the critical locus of the map $\mathcal{A} \circ \pi$ except at the preimage of the singular points of $\cV (f)$, see \cite{L2}. The topology of $S(f)$ is described below.

\begin{Proposition}\label{prop:S(f)}
Assume that no irreducible component of $\cV(f)$ is a binomial cylinder. Then, $S(f)$ is a union of smoothly embedded circles in $\tilde{\cV}(f)$ having at worst transversal intersection points. Moreover, $S(f)$ is smooth, $i.e$ a disjoint union of circles, if and only if $\TG$ has no branching points on $\rp{1}$.
\end{Proposition}

In the case when $\cV(f)$ has a binomial cylinder as an irreducible component, Corollary \ref{cor:binom} implies that the latter is entirely contained in $S(f)$.

\begin{Theorem}\label{thm:sa}
Let $\Delta \subset \R^2$ be a bounded lattice polygon of dimension 2 such that $\vert \Delta \cap \Z^2 \vert \geq 4$. Then, $\cS_\Delta$ is a semi-algebraic set of codimension 1. Moreover, $\cS_\Delta$ is the closure of its maximal stratum.
\end{Theorem}
As a consequence, the complement of $\cS_\Delta$ in the set of smooth curve $\D_0$ is possibly disconnected and the topological pair $(\cV(f), S(f))$ may vary from one connected component to the other. In particular, the number $ b_0 \big( S(f) \big) $ of connected components of $S(f)$ may vary. We deduce straight from the definition \eqref{eq:critloc} and Proposition \ref{deglogGauss} that the latter satisfies 
\begin{eqnarray}
	1 \leq b_0 \big( S(f) \big) \leq vol(\Delta)
\label{eq:trivbd}
\end{eqnarray}
for any smooth curve $\cV(f) \in \D_0 \setminus \cS_\Delta$.

\begin{Theorem}\label{thm:upbd}
For any positive integers $k$ and $d$ satisfying $$ d+ d(d-1)/2 \leq k\leq d^2,$$ there exists a smooth curve $\cV(h) \subset \cp{2}$ of degree $d$ such that $S(h)$ is smooth and $b_0(S(h)) = k$. In particular, the upper bound of \eqref{eq:trivbd} is sharp when $\Delta$ is the standard $d$-simplex. 
\end{Theorem}
In section \ref{sec:top}, we discuss the sharpness of the lower bound of \eqref{eq:trivbd} and motivate the following.

\begin{Conjecture}\label{conj}
Let $\Delta \subset \R^2$ be a bounded lattice polygon of dimension 2. Then, for any $[f] \in \D_0 \setminus \cS_\Delta$, we have
\begin{eqnarray}
	g_\Delta +1 \leq b_0 \big( S(f) \big)
\label{eq:lowerbd}
\end{eqnarray}
and the bound is sharp.
\end{Conjecture}

Let us briefly comment Conjecture \ref{conj}. First, the sharpness of the lower bound suggested here is guaranteed by the existence of Simple Harnack curves for any toric surface $\Xd = \Xd_\Delta$, see \cite[Lemma 5 and Corollary A4]{Mikh}. In the case $X=\cp{2}$, the technique used to prove Theorem \ref{thm:upbd} leaves unfortunately a gap between $g_\Delta+1 = d(d-1)/2 +1$ and $ d(d-1)/2+d$. We claim nevertheless that the remaining cases can be constructed using slightly different techniques.

The topology of $S(f)$ is closely related to the relative position of the branching divisor $B_{\TG_f}$ with respect to $\rp{1} \subset \cp{1}$.  While proving Theorems \ref{thm:sa} and \ref{thm:upbd}, the Lyashko-Looijenga mapping $\ly$ comes naturally into play.

\begin{Theorem}\label{thm:llalg}
Let $\Delta$ satisfy the assumptions of Theorem \ref{thm:sa}. Then, the map $\ly : \LL  \dashrightarrow \cp{m}$ is algebraic.
\end{Theorem}

In Section \ref{sec:top}, we prove Theorem \ref{thm:upbd} by perturbing union of well chosen lines. To this aim, we study the possibility of extending the Lyashko-Looinjenga mapping to nodal curves. Before stating our next result, we need some extra terminology.

For any $f \in \mathcal{D}$, denote by $I_f^{-}$ the set of binomial cylinders among the connected component of $\tilde{\cV}(f)$ and by $I_f^{+}$ its complement. Recall by Corollary \ref{cor:binom} that $\TG_f$ is constant exactly on components that are in $I_f^{-}$. For any $\mathcal{C} \in I_f^-$, we consider $\TG_f ( \mathcal{C} )$ as a degree $1$ divisor on $\cp{1}$. Finally, for any connected component $\mathcal{C}$ of $\tilde{\cV}(f)$, define $s(\mathcal{C}):= \sum_{p\in \mathcal{C}} p$ such that $\pi(p) \in \cV(f)$ is a singular point.
 
\begin{Theorem}\label{thm:extLL}
The formula 
\begin{equation}\label{eq:extLL}
\ly(f) := \sum_{\mathcal{C} \in I_f^-} ( 3 \cdot deg( s(\mathcal{C}))  -2 ) \cdot \TG_f ( \mathcal{C} ) \hspace{3cm}
\end{equation}
$ \hspace{5cm} \displaystyle + \sum_{\mathcal{C} \in I_f^{+}} (\gamma_{f})_\ast \big( ( R_{\TG_f} )_{|_\mathcal{C}} \big) + 3 \cdot (\TG_f)_\ast ( s(\mathcal{C}) )$\\
provides an algebraic extension of the map $\ly$ to the set of nodal curves $\D$. 
\end{Theorem}

\section{The Lyashko-Looijenga mapping}\label{sec:LL}
Recall that the Lyashko-Looijenga mapping $\ly$ is defined as
\[
\begin{array}{rcl}
\ly : \LL & \dashrightarrow & \cp{m}\\
\left[f\right] & \mapsto & B_{\gamma_f}
\end{array}
\]
on the set of projective classes of Laurent polynomials for which $\cV(f)$ is smooth and intersect $\Xd^\infty$ transversally.

\subsection{Algebraicity of the $\mathcal{LL}$ map}

The goal of this section is to prove Theorem \ref{thm:llalg}. To show the algebraicity of $\ly$, we make use of the nice theory of  resultants \`{a} la \cite{GKZ}.

Let $\Delta\subset \R^2$  be a bounded lattice polygon of dimension $2$ and define $\Xd:=\Xd_\Delta$, $\Li:=\Li_\Delta$. Recall that any Laurent polynomial $f$ with support in $\Delta\cap\Z^2$ can be viewed as a point of the vector space $H^0(\Xd,\Li)$ via the projective embedding $\Xd \hookrightarrow \cp{m}$ induced by $\Delta$. According to \cite[Chapter 3.2.B]{GKZ}, there exists a polynomial $R_\Delta (f_1, f_2, f_3)$ on $\big( H^0(\Xd,\Li)\big)^3$ homogeneous in each factor and  such that
\begin{center}
$ R_{\Delta} (f_1 , f_2 , f_3 ) = 0 $ $\Leftrightarrow$  $ \bigcap_k \cV(f_k)  \subset \Xd$ is non-empty.
\end{center}
\noindent
Now, fix some $ f_1 $ and $ f_2 $ such that $ \cV(f_1) $ and $\cV(f_2)$ intersect transversally in $\Xd$. Let $d:= \vert \cV(f_1) \cap \cV(f_2) \vert$. Then, $ R_\Delta$ regarded as a polynomial in $f_3$ splits into a product of $d$ linear forms. Indeed, $ R_\Delta$ vanishes if and only if $ f_3 $ passes through one of the points of $\cV(f_1) \cap \cV(f_2)$. Each of these points corresponds to a linear form in the dual space, whose product is $R_\Delta$, see \cite[Ch.4, Sec.2]{GKZ}. The space of homogeneous polynomials of degree $d$ on  $ \cp{m} $ that splits into linear forms is an algebraic subvariety of $ \cp{N} $ with $ N=\binom{d}{m+1} $. It parametrizes the space $ G(1,d,m) $ of zero cycles of degree $d$ on $ \cp{m} $ which is also isomorphic to $ Sym_d ( \cp{m} )$, see \cite[Ch.4, theorems 2.2 and 2.12]{GKZ}.

\begin{proof}[Proof of Theorem \ref{thm:llalg}] Consider  the set $ \mathcal{G} \subset \LL \times \cp{1} $ of pairs $\big( f , [u:v] \big) $ such that $[u:v]$ is a branching point of  $\TG_f$. Using a local parametrization of $\cV(f)$, we see that these are exactly the pairs $ \left( f, [u:v] \right) $ for which the two level sets $ \left\lbrace f= 0 \right\rbrace $ and $ \left\lbrace \gamma_{f} = [u:v] \right\rbrace $ do not intersect transversally.
Note that the level set $ \left\lbrace \gamma_{f} = [u:v] \right\rbrace $ can be alternatively defined as the curve $\cV(\gamma_{f,[u:v]})$ where
\[ \gamma_{f,[u:v]} (x,y) = x\dfrac{\partial f}{ \partial x} (x,y) \cdot v - y\dfrac{\partial f}{ \partial y} (x,y) \cdot u. \]
As the support of the latter polynomial is included in $\Delta\cap\Z^2$, the curve $\cV(\gamma_{f,[u:v]})$ can be realized as a hyperplane section in $ \LL$. Thanks to Proposition \ref{deglogGauss}, we know that the curves $\cV(f)$ and $\cV(\gamma_{f,[u:v]})$ always intersect in $vol(\Delta)$ many points, counted with multiplicities. Recall that $\D_0$ denotes the set of smooth curves in $\LL$. Then, we can consider the following algebraic map 
\[ 
\begin{array}{rcl}
				\phi :  \D_0  \times \cp{1}  	& \rightarrow 	& G(1, d, m ) \\
				\big( f, [u:v] \big)			& \mapsto		&   R_\Delta \left( f, \gamma_{f,[u:v]}, \, . \, \right)
\end{array},
\]
where $d:= vol(\Delta)$ and $m:= \vert \Delta\cap\Z^2 \vert$.
This map associates to any $\big( f, [u:v] \big)$ a product of  $d$ linear forms on $\LL$: they correspond to the constraint of passing through one of the $d$ points of $\cV(f) \cap \cV(\gamma_{f,[u:v]})$. Now, observe that $ \mathcal{G} = \phi^{-1} ( \mathcal{H} ) $, where $\mathcal{H}$ is the hypersurface of $Sym_d ( \cp{m} ) \simeq G(1, d, n ) $ corresponding to the $d$-tuples  for which at least two coordinates coincide. It corresponds exactly to the singular locus of $Sym_d ( \cp{m} )  $, see \cite[Ch. 4, Sect. D]{GKZ}. Therefore, $\mathcal{H}$ is algebraic and so is $ \mathcal{G} $. Now, the assumption $\vert \Delta\cap\Z^2\vert\geq 4$ together with the Riemann-Hurwitz formula guarantee that $ \TG_f$ has branching points for any $f$, implying that the projection $\mathcal{G} \rightarrow \D_0$ is surjective. We deduce that $\mathcal{G}$ is an algebraic hypersurface in $\D_0 \times \cp{1}$. It has to be given by a bihomogeneous polynomial $P$ on $ \LL \times \cp{1}$. Now, observe that the coefficients of $P (f, \, . \, )$ are polynomial in the coefficients of $f$ and that its roots give the divisor $\ly(f)$. This proves that $\ly$ is algebraic on the open subset of $\D_0$ of curves intersecting $\Xd^\infty$ transversally.
\end{proof}

\begin{Remark} Up to projective equivalence, we can assume that the coefficients of $\ly(f)$ regarded as polynomials in the coefficients of $f$, has no common factors. In such case, common zeroes of the coefficients of $\ly(f)$ correspond to polynomials $f$ for which $\ly$ is a priori not well defined.
\end{Remark}

\begin{Remark}\label{rem:ll} 
The polynomial $\ly(f)$ can be computed indirectly by computing the resultant $ R_\Delta \left( f, \gamma_{f,[u:v]}, g \right)$ where $$g:= \det \big( grad \, f, grad \, \gamma_{f,[u:v]} \big).$$ Indeed, $g$ vanishes at common root $(z,w) \in \ttor$ of $f$ and $\gamma_{f,[u:v]} $ if the level sets $ \left\lbrace f= 0 \right\rbrace $ and $ \left\lbrace \gamma_{f} = [u:v] \right\rbrace $ do not intersect transversally. Nevertheless, $\ly(f)$ is only a factor of the latter resultant as it might have extra roots corresponding to non-geometric solutions of the system $f=\gamma_{f,[u:v]}=g=0$ on $\Xd^\infty$, see the examples in section \ref{sec:ex}.
\end{Remark}

\subsection{Extension of $\mathcal{LL}$ to nodal curves}

In this section, we prove Theorem \ref{thm:extLL}.
It is not clear a priori that $\ly$ extends to the whole linear system $\LL$ since the degree of $\TG_f$ can drop, for two reasons.

As discussed in Remark \ref{rmk:degree}, $deg(\TG_{f_t})$ drops along a family $\{f_t\}_t$ for which several points of $\cV(f_t)^\infty$ merge together. If $p\in \cV(f_t)^\infty$ and $m_p$ denotes the intersection multiplicity of $\cV(f)$ and $\Xd^\infty$ at $p$, we can check that the formula
\[ B_{\TG_f} + \sum_{p\in \cV(f_t)^\infty} (m_p-1)\cdot \gamma_f(p) \]
provides a continuous extension of $\ly$ to curves with non transversal intersection with $\Xd^\infty$. Hence, we can restrict to curves that are transverse to $\Xd^\infty$ while proving Theorem \ref{thm:extLL}, and this without any loss of generality.

More serious issues happen when $\cV(f)$ becomes singular. Along a degeneration, we can observe that several points of $R_{\gamma_f}$ converge to the singularities. We need to be able to determine the quantity of such point for each new singularity, keep track of their image in $\cp{1}$ and show that their limit does not depend on the degeneration. Below, we proceed for any curves having only double points as singularities.

We now study the behaviour of the logarithmic Gau\ss{} map along 1-parametric families of curves locally isomorphic to
\begin{eqnarray}\label{eq:locmod}
zw=t
\end{eqnarray}
near the origin of $\C^3$.
In what follows, we consider holomorphic maps $F : U \rightarrow \LL $ from an open disc $U\subset \C$, and define $f_t:=F(t)$ for any $t\in U$. Recall by that $\D_{l+k} \subset \overline{\D_l}$ is locally isomorphic to the intersection of $k$ coordinate hyperplanes  $\left\lbrace z_1=0\right\rbrace \cap ... \cap \left\lbrace z_k=0\right\rbrace \subset \C^N$ where $N=dim_\C \,  \D_l$, see e.g \cite[Proposition (2.11)]{Tan}.

\begin{Definition}\label{def:ltrans}
Let $l$ be a strictly positive integer. A map $F : U \rightarrow \LL $ is $l$-transverse if there exist finitely many points $t_1$, ..., $t_n\in U$ and strictly positive integers $k_1$, ..., $k_n$ such that
\begin{itemize}
\item[$\ast$] $F(U\setminus \left\lbrace t_1, ..., t_n \right\rbrace) \subset \D_l$,
\item[$\ast$] $F(t_j)\in \D_{l+k_j}$, $1 \leq j \leq n$,
\item[$\ast$] $F(U)$ intersects transversally each of the $k_j$ smooth hypersurfaces defining $\D_{l+k_j} \subset \D_l$ at $t_j$, $1 \leq j \leq n$.
\end{itemize}
In the sequel, we refer to the $t_j$'s as the special points of $F$. The map $F$ induces a smoothing of $k_j$ of the $l+k_j$ nodes of $\cV(f_{t_j})$. We refer to them as the unstable nodes of $\cV(f_{t_j})$.
\end{Definition}

Consider an $l$-transversal family $F : U \rightarrow \LL $, and choose a special point $t_j$ of $F$ and an unstable node $p \in \cV(f_{t_j})$. It follows from \cite[Proposition 2.1]{ACG2} that $\left\lbrace f_t \right\rbrace_{t \in U}$ is isomorphic to \eqref{eq:locmod} in a neighbourhood of $(p,t_j)$. Recall also that the family $\left\lbrace f_t \right\rbrace_{t \in U}$ gives rise to the vanishing cycles $\delta_{p,t} \subset \cV(f_t)$ defined up to isotopy. In the model \eqref{eq:locmod}, $\delta_{p,t}$ is given by the circle $\{ zw=t \} \cap \{ \vert z \vert = \vert w \vert \}$, see \cite[Chapter X,  \S 9]{ACG2} for more details.

\begin{Definition}\label{def:cylinder}
If $p_1, \, p_2 \in \tilde{\cV}(f_{t_j})$ denotes the two points above $p$ and $v_1, \, v_2 \in \cp{1}$ their respective image by $\TG_{f_{t_j}}$, a cylinder $C \subset  \cp{1}$ is coherent with $p$ if it is given as the complement of two non intersecting open discs containing $v_1$ and $v_2$ respectively.
\end{Definition}

\begin{Proposition}\label{prop:dbcov}
Given an $l$-transversal family $F : U \rightarrow \LL $, a special point $t_j$ of $F$, an unstable node $p \in \cV(f_{t_j})$, and a cylinder $C \subset  \cp{1}$ coherent with $p$, there exist an $\varepsilon >0$ and a continuous family $\left\lbrace C^p_{f_t} \right\rbrace_{0 < \vert t-t_j \vert < \varepsilon}$ of cylinders $C^p_{f_t} \subset \cV(f_t)$ such that
\begin{enumerate}
\item[$i)$] $\gamma_{f_t} : C^p_{f_t} \rightarrow C$ is an unramified double covering,
\item[$ii)$] $C^p_{f_t}$ is homotopic to the vanishing cycle $\delta_{p,t}$ and converges to $p$ as $t$ converges to $t_j$.
\end{enumerate}
Moreover, the family  $\left\lbrace C^p_{f_t} \right\rbrace_{0 < \vert t-t_j \vert < \varepsilon}$ is unique.
\end{Proposition}

The strategy of the proof is to reduce to the following model. Consider the curve 
$$ \mathcal{C}_\theta :=\left\lbrace (z,w) \in \C^2 \, \big| \, zw=\theta \right\rbrace$$ 
 for any unimodular $ \theta \in \C$. For any chosen determination of $\sqrt{\theta}$, consider the parametrization of  $\mathcal{C}_\theta$ by the map $ s \mapsto \big( \sqrt{\theta} s, \sqrt{\theta} s^{-1} \big)$ defined on $\C^\ast$. The Gau\ss{} map on $\mathcal{C}_\theta$ is given by $g_\theta(s) = \big[ \sqrt{\theta} s^{-1}, \sqrt{\theta} s \big]$. Seen through the appropriate affine chart of $\cp{1}$, $g(s)$ is the unramified double covering $s\mapsto s^2$ onto $\C^\ast$.  In particular, for any constant $M>1$, the restriction of $\mathcal{C}_\theta$ to the closed ball $\overline{\mathcal{B}}(0,M)$ is mapped onto $\left\lbrace s \in \C \, \vert \, M^{-2} \leq \vert s \vert \leq M^2 \right\rbrace$ by $g_\theta$.
 
In the proof below, we show that the model $g_\theta : \mathcal{C}_\theta \rightarrow \cp{1}$ arise as the limit of the map $\gamma_{f_t} : \cV(f_t) \rightarrow \cp{1}$ restricted to a neighbourhood of $p$, when $t$ converges to $t_j$ along a radius of $U$. We then construct the cylinder $C^p_{f_t}$ by parallel transporting $g_\theta^{-1}(C)$ when the parameter $t$ travels away from $t_j$ along the specified radius. To this end, we give here some recall on parallel transport. Let $M$ be a fibre bundle of fibre $F$ with base $B$ and let $\pi : M \rightarrow B$ be the projection on the base. A connection $\nabla$ is a distribution of subspace $H_x \subset T_xM$ for any $x \in M$ such that $$H_x \oplus T_x \; \big(\pi^{-1}(\pi(x))\big)=T_xM$$ and $H_x$ depends smoothly on $x$. For any smooth path $\rho : \left[0,1\right] \rightarrow B$ from a point $p$ to a point $q$, the connection $\nabla$ allows to lift the tangent vector field of $\rho$ to a vector field $X_\rho$ in $\cup_{t \in \left[0,1\right]} \pi^{-1} \big( \rho(t) \big) $ everywhere transverse to any fiber $\pi^{-1} \big( \rho(t) \big) $. The parallel transport of a point $x \in \pi^{-1}(p)$ along $\rho$ is the end point in $\pi^{-1}(q)$ of the unique integral curve of $X_\rho$ starting at $x$. When the fiber $F$ is non-compact, that will be our case below, it might happen that the integral curve escape to infinity for a time $t<1$. As we are interested in a local statement, we can always assume that $q$ is close enough to $p$ so that the parallel transport is well defined. The parallel transport of any compact subset $K \subset \pi^{-1}(p)$ along $\rho$ is defined as the union of the parallel transports of every point $x\in K$ in $\pi^{-1}(q)$, provided that $q$ is close enough to $p$. By Cauchy-Lipschitz Theorem, the solutions of a first-order differential equation depend smoothly in the initial conditions. Then the parallel transport of $K$ along $\rho$ is diffeomorphic to $K$ itself. We refer to \cite[Sections 24.1 and 24.2]{DFN} and \cite[Proposition 1.17]{BGV} for more details.

\begin{proof}[Proof of Proposition \ref{prop:dbcov}]
Consider logarithmic coordinates $(z,w)$ in a neighbourhood of $p$ and denote by $h_t:=f_t \circ \exp$. In the $(z,w)$-plane, the logarithmic Gau\ss{} map is the standard Gau\ss{} map, namely the projectivization of the gradient $\nabla h_t$ on $\cV(h_t)$. As linear transformations commute with the Gau\ss{} map, we can assume that $\Log(p)=(0,0)$ and that the two branches 
of $\cV(h_0)$ at the origin are given by the two coordinate axes. After the induced Moebius transformation on $\cp{1}$, the cylinder $C$ is a non-contractible subset of $\cp{1} \setminus \left\lbrace 0, \infty \right\rbrace$. Translate the disc $U$ in $\C$ so that $t_j=0$. Then the $l$-transversality of $\left\lbrace f_t \right\rbrace_{t \in U}$ ensures that the monomial $t$ appears with a non zero coefficient in the Taylor expansion of $h_t$. Dividing $h_t$ by a constant and applying a linear rescaling to  $t$,  $h_t$ admits the following Taylor expansion
\begin{eqnarray}\label{eq:taylor}
 h_t(z,w)	=  zw - t + \displaystyle \sum_{a + b + 2 c > 2} \alpha_{(a, b, c)} z^a w^b t^c
\end{eqnarray} 
for some $\alpha_{(a, b, c)} \in \C$. Consider now the change of coordinates $\varphi_t (z,w):=\vert t \vert^{-\frac{1}{2}} (z,  w)$ and denote by $\cV_t := \varphi_t \big( \cV(h_t)\big)$. Notice that $\varphi_t$ preserves the Gau\ss{} map. After applying the change of coordinates $\varphi_t$ in \eqref{eq:taylor} and dividing by $\vert t \vert$, we deduce that the curve $\cV_t$ is the vanishing locus of a function of the form
\begin{eqnarray}\label{eq:gauss1}
 zw - t/\vert t \vert  + \vert t \vert^{\frac{1}{2}} \cdot R_t(z,w)  
\end{eqnarray} 
where the remainder $R_t(z,w)$ is analytic and depends continuously on $t$. The Gau\ss{} map on $\cV_t$ is then given by 
\begin{eqnarray}\label{eq:gauss2}
 g_t(z,w) = \left[ w + \vert t \vert^{\frac{1}{2}} \cdot \partial_z R_t(z,w), z + \vert t \vert^{\frac{1}{2}} \cdot \partial_w R_t(z,w) \right].
\end{eqnarray} 
Denote by $\mathcal{G}_t$  the graph of $g_t$ over $\cV_t$ and by $\mathcal{G}_\theta$ the graph of $g_\theta$ over $\mathcal{C}_\theta$. As a consequence of \eqref{eq:gauss1} and \eqref{eq:gauss2}, the restriction of $\mathcal{G}_t$ to any compact neighbourhood $K\subset \C^2$ of the origin  converges in Hausdorff distance to the restriction of $\mathcal{G}_\theta$ to $K$ when $t$ converges to zero and $t/\vert t \vert$ converges to $\theta$. We conclude that the continuous family of holomorphic curves $\left\lbrace \mathcal{G}_t \cap K \right\rbrace_{t \neq 0}$   can be compactified over the real oriented blow up $U_\circ$ of $U$ at $0$ by the family $\left\lbrace \mathcal{G}_\theta \cap K\right\rbrace_{\vert \theta \vert =1}$. Fix therefore $K$ such that the Gau\ss{} map $g_\theta$ on the curve $\mathcal{C}_\theta \cap K$ surjects on the cylinder $C$ for any unimodular complex number $\theta$. As we pointed out before the proof, $g_\theta^{-1}(C)$ is a cylinder that is  a double cover of $C$. Let us show that for some $\varepsilon>0$, we can parallel transport $g_\theta^{-1}(C)$ to $\cV_t \cap K$ along the level sets of $g_t$ for any $\vert t \vert < \varepsilon$. The fibre bundle that we consider is 
$$ M:= \bigcup_{t \in U_\circ, \, \vert t \vert < \varepsilon} \cV_t \cap K \; \longrightarrow \; \left\lbrace t \in U_\circ, \, \vert t \vert < \varepsilon \right\rbrace$$
where $\varepsilon$ will be fixed latter. 
The connection $\nabla$ at a point $x \in \cV_t \subset M$ such that $g_t(x)=v$ is defined as the intersection of the tangent spaces of $\cV_t$ and the level set $\left\lbrace (z,w,t) \in K \times U_\circ \, \vert \, g_t(z,w)=v \right\rbrace$ at $x$. For $\nabla$ to be well defined, this intersection has to be transverse and the resulting line in $T_x \big( K \times U_\circ \big)$ has to have a non-zero coordinate in $t$. 
The last requirement is guaranteed by the fact that $\cV_t$ is smooth for any $t$. For the first requirement, it is enough to show that $\cV_t$ and the level set $\left\lbrace (z,w) \in K \, \vert \, g_t(z,w)=v \right\rbrace$ intersect transversally in $K$. This can be guaranteed for any $v\in C$ and any $t$ such that $\vert t \vert < \varepsilon$, for a small enough $\varepsilon>0$. Indeed, $\left\lbrace (z,w) \in K \, \vert \, g_t(z,w)=v \right\rbrace$ and $\cV_t$ intersect transversally if $v\in \cp{1}$ is not an isotropic direction for the Hessian of \eqref{eq:gauss1} 
\begin{eqnarray}\label{eq:gauss3} 
\begin{pmatrix}0 & 1 \\ 1 & 0 \end{pmatrix} + \vert t \vert^{\frac{1}{2}} \cdot Hess \, R_t(z,w). 
\end{eqnarray} 
We can take $\varepsilon$ small enough so that the isotropic direction of \eqref{eq:gauss3} are close to the one of the matrix on the left, namely $0$ and $\infty$. Hence, these isotropic directions lie outside of $C$. Therefore we can define the parallel transport $C_t$ of $g_\theta^{-1}(C)$ to $\cV_t$ along the radius $\left\lbrace s \in U_\circ \, \vert \, arg(s)=\theta \right\rbrace$. We might have to consider a smaller $\varepsilon$ in order for $C_t$ to stay in $K$, for any $\theta \in S^1$ and $t$ such that $\vert t \vert < \varepsilon$ and $arg(t)=\theta$. As $C_t$ is given by the parallel transport of the cylinder $g_\theta^{-1}(C)$ along the level set of $g_t$ and $g_\theta$ is a double covering, it follows that $C_t$ is a cylinder, that $C_t=g_t^{-1}(C)$ and finally that $g_t : C_t \rightarrow C$ is a double covering. Taking $C^p_{f_t} \subset \cV(f_t)$ to be the preimage of $C_t$ by $\varphi_t \circ \Log$ proves $i)$.

The fact that $C^p_{f_t}$ converges to $p$ for $t \to 0$ is clear since $\Log C^p_{f_t} \subset \vert t \vert^{\frac{1}{2}}\cdot K$. $C^p_{f_t}$ is either contractible or homotopy equivalent to a vanishing cycle. In order to prove $ii)$, it remains to prove that $C^p_{f_t}$ is not contractible in $\cV(f_t)$. On one side, the fibration $\left\lbrace C^p_{f_t} \right\rbrace_{t\in \rho}$ over any non trivial loop $ \rho \in U\setminus \{0\}$  is homeomorphic to the non trivial fibration $\left\lbrace C_\theta \right\rbrace_{\vert \theta \vert =1}$, see \cite[Chapter X, \S 9]{ACG2}. Assume that $C^p_{f_t}$ is contractible for some $t$, then it is for any $0<\vert t \vert < \varepsilon$ and the fibration $\left\lbrace C^p_{f_t} \right\rbrace_{t\in \rho}$ over any loop $\rho$ is trivial. This is a contradiction and proves $ii)$.

The uniqueness of the family of cylinders $C^p_{f_t}$ follows from the point $ii)$ and Proposition \ref{deglogGauss}. As $C^p_{f_t}$ is a double covering of $C$ and as $C^p_{f_t}$ vanishes if $t \to 0$ by $i)$ and $ii)$, the degree of $\TG_{f_t}$ drops by twice the number of such families of cylinders when $t$ goes to 0. By Proposition \ref{deglogGauss}, it should only decrease by twice the number of unstable nodes of  $\cV(f_{t_j})$. It implies that there is only one such family of cylinders $C^p_{f_t}$ for any unstable node $p \in \cV(f_{t_j})$.
\end{proof}
\noindent

The next step in the proof of Theorem \ref{thm:extLL} is the study of the limit of the ramification divisor of $\TG_{f_t}$ along continuous paths $\left\lbrace f_t \right\rbrace_{t>0} \subset \D_l$ that converge to a fixed $f_0 \in \D_{l+k}$, for strictly positive integers $l$ and $k$. The strategy is to cut out $\cV(f_t)$ into appropriate pieces by removing the cylinders we have constructed in the previous proposition for all the unstable nodes of $\cV(f_0)$. As any such cylinder contains no ramification point, we can study the limit of the ramification divisor of $\TG_{f_t}$ on each connected component of the complement of these cylinders in $\cV(f_t)$. In order to make sense of these cylinders, we will construct a foliation of $\overline{\D_l}$ by $l$-transverse discs in a neighbourhood of $f_0$.

We proceed as follows. Let $N=dim_\C \, \D_l$. Choose an open neighbourhood $\mathcal{B}_1 \subset \overline{\D_l}$ of $f_0$ together with holomorphic coordinates $(z_1$, ..., $z_N)$ centred at $f_0$ such that $\mathcal{B}_1$ is the unit ball in $\C^N$ and $\D_{l+k} \cap \mathcal{B}_1 = \left\lbrace z_1=0\right\rbrace \cap ... \cap \left\lbrace z_k=0\right\rbrace $. Therefore any disc in $\mathcal{B}_1$ given as the intersection of a line of slope $(1,...,1)$ with $\mathcal{B}_1$ is $l$-transverse. All such discs are parametrized by the hyperplane $H:= \left\lbrace z_1+...+z_N = 0 \right\rbrace$. Denote by $U_z$ the $l$-transverse disc associated to $z\in H$. Note first that there exists a poly-disc  $V\subset H$ centred at the origin for which the disc $U_z$ intersects all the hyperplanes $\left\lbrace z_j=0\right\rbrace$, for any $1\leq j \leq k$ and $z \in V$.  Hence, for any $z\in V$,  the disc $U_z$ has $k$ unstable nodes distributed among $n\leq k$ special points all merging to $0\in U_0$ when $z$ approaches $0\in H$, see Definition \ref{def:ltrans}. Define now  the open set $ \mathcal{B}_2:= \cup_{z \in V} U_z \subset \mathcal{B}_1$. For any $f\in \mathcal{B}_2$ denote by $z(f)$ the projection of $f$ on $V$,  by $t_{j,z} \in U_z \cap \{z_i=0\}$ and by $p_{j,z}$ the unstable node at the special point $t_{j,z}$ corresponding to $\left\lbrace z_j=0\right\rbrace$ (it may happen that $t_{j,z}=t_{l,z}$ for $j\neq l$). Finally, denote by $U_{j,z} \subset U_z$ the disc $\left\lbrace \vert t-t_{j,z}\vert < \varepsilon \right\rbrace$ where $\varepsilon$ is given by Proposition \ref{prop:dbcov}. Now, for a chosen unstable node of $U_z$, the parameter $\varepsilon$ of the Proposition \ref{prop:dbcov} varies continuously in $z$. As $0$ is the unique special point of $U_0$, it implies that there is an open neighbourhood $\mathcal{B}_3 \subset \mathcal{B}_2$ of $f_0$ such that 
\[  \bigcap_{1\leq j \leq k} U_{j,z(f)} \neq \emptyset \]
for any $f\in\mathcal{B}_3$. By Proposition \ref{prop:dbcov}, we can define the cylinders $C^{p_{j,z(f)}}_f$ for any $f\in\mathcal{B}_3\cap \D_l$ and $1\leq j \leq k$.

We now consider a continuous paths $\left\lbrace f_t \right\rbrace_{t>0} \subset \D_l$ converging to $f_0 \in \D_{l+k}$. The above discussion allows us to define
\[ \cV(f_t)^\circ := \cV(f_t) \setminus  \; \bigcup_{1\leq j \leq k}  C^{p_{j,z(f_t)}}_{f_t} \]
for any $\vert t \vert < \delta$ and some $\delta<0$. By the point $ii)$ of Proposition \ref{prop:dbcov},  the connected component of $\cV(f_t)^\circ$ and the one of $\tilde{\cV}(f_0)$ are in bijective correspondence via the path $\left\lbrace f_t \right\rbrace_{t>0}$. For any connected component $\mathcal{C} \subset \tilde{\cV}(f_0)$, denote by $\mathcal{C}^\circ_t$ the corresponding connected component of $\cV(f_t)^\circ$. Recall finally that the divisor $s(\mathcal{C}) $ on $\mathcal{C}$ is the sum of the point of $C$ that project down to a double point of $\cV(f_0)$. 

\begin{Lemma}\label{lem:limram}
For any connected component $\mathcal{C} \subset \tilde{\cV}(f_0)$, either $\gamma_{f_0}$ is constant on $\mathcal{C}$ and  then 
\[deg ( R_{\gamma_{f_t}} )_{|_{\mathcal{C}_t^\circ}} = 3 \cdot deg(s(\mathcal{C})) -2, \]
or $\gamma_{f_0}$ is not constant  on $\mathcal{C}$ and we have 
\[ \lim_{t \rightarrow \infty} ( R_{\gamma_{f_t}} )_{|_{\mathcal{C}_t^\circ}} = ( R_{\gamma_{f_0}} )_{|_\mathcal{C}} + 3 \cdot s(\mathcal{C}). \]
\end{Lemma}

\begin{proof}[Proof of Lemma \ref{lem:limram}]
In the first case, $\mathcal{C}$ has to be the compactification of a binomial cylinder that is a topological sphere, see Corollary \ref{cor:binom}. It implies that $\mathcal{C}_t^\circ$ is a sphere with $deg(s(\mathcal{C}))$ many discs removed. Let $\left\lbrace v \right\rbrace \subset \cp{1}$ be the image of $\mathcal{C}$ by $\gamma_{f_0}$. When defining $\mathcal{C}_t^\circ$ via Proposition \ref{prop:dbcov}, we choose the cylinder $C\subset \cp{1}$  coherent with $p$ for each point $p \in s(\mathcal{C})$ so that they all share the same boundary component, let us say $\partial$, around $v$.
Define $D \subset \cp{1}$ to be the disc bounded by $\partial$ that contains $v$. As $\gamma_{f_0}$ is constant on $\C$, the map $\gamma_{f_t}  :\mathcal{C}_t^\circ \rightarrow \cp{1}$ has degree 0 and $\gamma_{f_t} (\partial \mathcal{C}_t^\circ) = \partial$. Hence, $\gamma_{f_t} (\mathcal{C}_t^\circ)=D$. The degree of $\gamma_{f_t}  :\mathcal{C}_t^\circ \rightarrow D$ can be computed by the contribution on its boundary. By Proposition \ref{prop:dbcov}, each boundary component of the source contributes by two. The Riemann-Hurwitz formula \eqref{eq:RHformula} yields
\[
\begin{array}{rcl}
deg ( R_{\gamma_{f_t}} )_{|_{\mathcal{C}_t^\circ}} &=& deg ( \gamma_{f_t} : \mathcal{C}_t^\circ \rightarrow D ) - \chi ( \mathcal{C}_t^\circ )\\
& & \\
& = & 2 \cdot deg(s(\mathcal{C}) ) - ( 2 - deg(s(\mathcal{C}) ) ),
\end{array}
\]
implying the first part of the statement.

For the second case, let us first investigate the limit of the ramification divisor at a point $q \in s(\mathcal{C})$. Let $k \geq 0$ be the order of the divisor $ R_{\TG_{f_0}} $ at $q$. Then, $\TG_{f_0}$ is locally given by $z \mapsto z^{k+1}$ around $q$ and the image of any sufficiently small loop $\rho$ around $q$ winds ($k+1$)-times around $v \in \cp{1}$. Any small deformation $\rho_t$ of $\rho$ in $\mathcal{C}_t^\circ$ (for small $t$) has the same winding property, whereas the boundary circle of $\mathcal{C}_t^\circ$ corresponding to $q$ winds twice around $v$ via $\gamma_{f_0}$, by $i)$ in Proposition \ref{prop:dbcov}. These two loops bound a cylinder in  $\mathcal{C}_t^\circ$ that covers $(k+3)$-times a small disc around $v\in \cp{1}$. The Riemann-Hurwitz formula \ref{eq:RHformula} implies that this cylinder contains $(k+3)$ ramification points. It implies that
\begin{equation}
\label{eq:+3}
 \left( \lim_{t\rightarrow0} R_{\gamma_{f_t}} \right) (q) = R_{\gamma_{f_t}} (q) +3.
\end{equation}
Let us now investigate the limit of the ramification divisor away from the support of $s(\mathcal{C})$. For each point $q \in s(\mathcal{C})$ consider a small ball $\mathcal{B}_q \subset \ttor $ that contains no ramification point of $\TG_{f_0}$ except possibly $q$ itself. Denote by $\mathcal{C}^{\circ \circ}_t := \mathcal{C}^\circ_t \setminus \cup_{q\in s(\mathcal{C})} \mathcal{B}_q$ (respectively $\mathcal{C}^{\circ \circ} := \mathcal{C} \setminus \cup_{q\in s(\mathcal{C})} \mathcal{B}_q$). For $t$ small enough, no ramification point of $\gamma_{f_t}$ meets $\partial \mathcal{B}_q$ for all $q\in s(\mathcal{C})$. For such $t$'s the degree of $R_{\gamma_{f_t}}$ restricted to $\mathcal{C}^{\circ \circ}_t$ does not drop when $t$ approaches 0. As the family of holomorphic maps $\gamma_{f_t} : \mathcal{C}^{\circ \circ}_t \rightarrow \cp{1} $ is continuous in $t$, it implies that
\[ \lim_{t \rightarrow 0} ( R_{\gamma_{f_t}} )_{|_{\mathcal{C}^{\circ \circ}_t}} = ( R_{\gamma_{f_0}} )_{|_{\mathcal{C}^{\circ\circ}}}.\]
Together with (\ref{eq:+3}), the latter implies the result.
\end{proof}

\begin{proof}[Proof of Theorem \ref{thm:extLL}.] 
 Recall our candidate extension for the $\ly$ map, namely 
$$
\ly(f) := \sum_{\mathcal{C} \in I_f^-} ( 3 \cdot deg( s(\mathcal{C}))  -2 ) \cdot \TG_f ( \mathcal{C} ) \hspace{3cm}
$$
$ \hspace{5cm} \displaystyle + \sum_{\mathcal{C} \in I_f^{+}} (\gamma_{f})_\ast \big( ( R_{\TG_f} )_{|_\mathcal{C}} \big) + 3 \cdot (\TG_f)_\ast ( s(\mathcal{C}) ).$\\
First, notice that $\ly(f)$ coincide with the original definition when $\cV(f)$ is smooth. Indeed, $I^-_f$ is empty whereas $I_f^{+}$ consists only of $\cV(f)$ itself, and $s(\cV(f))$ is empty as well. In such case, $\ly(f)$ is simply the push-forward of the ramifiaction divisor, namely the branching divisor of $\gamma_f$. According to this definition, the restriction of $\ly(f)$ to a fixed stratum $\D_l$ is continuous. Indeed, the term $\TG_f ( \mathcal{C} ) $ is constant and both $( R_{\TG_f} )_{|_\mathcal{C}}$ and $(\TG_f)_\ast ( s(\mathcal{C}) )$ depend algebraically on $f$. Now, we need to show continuity while passing from the $l$-th stratum to the $(l+k)$-th stratum for any $k>0$. For this purpose, consider a continuous path $\left\lbrace f_t \right\rbrace_{0<t\leq 1} \subset  \D_l $ converging to $f_0 \in \D_{l+k}$. We can apply Lemma \ref{lem:limram} and pass to the limit $t \to 0$ in the formula \eqref{eq:extLL}. The binomial cylinders of $I_{f_t}^-$ cannot decompose in several irreducible components when t approaches zero as they are defined by binomials of the form $z^aw^b-c$ with $a$ and $b$ coprime. In particular they cannot acquire additional double points. Their image by the logarithmic Gau\ss{} map is given by a rational vector independent of $t$, see Corollary \ref{cor:binom}. Hence, the first sum is constant. The limit of the second sum is explicitly given by Lemma \ref{lem:limram}. Combining the two summands, the result is $\ly(f_0)$. As the formula depends neither on $l$ nor on  $ \left\lbrace f_t \right\rbrace_{0<t\leq 1} $ but only on $f_0$, it induces a continuous extension of $\ly$ to $\D$. By Theorem \ref{thm:llalg} and the Riemann Extension Theorem, see \cite[Chapter 1,Section C]{GR}, the latter extension is algebraic.
\end{proof}

\section{Semi-algebraicity of the discriminantal set}\label{sec:sa}

\begin{proof}[Proof of Theorem \ref{thm:sa}] According to Proposition \ref{prop:S(f)}, a point $[f] \in \D_0$ belongs to $\cS_\Delta$ if and only if the univariate polynomial $\ly(f)$ on $\cp{1}$ has at least one real root. Let $m$ be the degree of $B_{\TG_f}$ or equivalently the degree of $\ly(f)$. Then, $\ly(f)$ is a point of $\cp{m} \simeq Sym_m(\cp{1})$. Now, we show that the space $\P$ of polynomials of degree $m$ with at least one real root is a semi-algebraic hypersurface in $\cp{m}$. Indeed, the intersection of $\P$ with the subspace $\left\lbrace (z_1,...,z_m) \in \C^m \big| \; z_i \neq z_j, \, \forall \,  i\neq j \right\rbrace \subset Sym_m(\cp{1})$  of polynomials with pairwise distinct roots is given by by the union of real hyperplanes
\[ \displaystyle \bigcup_{1\leq k \leq m} \left\lbrace (z_1,...,z_m) \in \C^m \big| \; z_i \neq z_j \text{ and } Im(z_k)=0\right\rbrace. \]
As $\P$ is the closure of the latter inside $\cp{m}$, it implies that it is a semi-algebraic hypersurface obtained as the closure of its codimension 1 stratum. By Proposition \ref{prop:S(f)} and Theorem \ref{thm:llalg}, $\cS_\Delta = \ly^{-1} ( \P ) \cap \D_0$ is a semi-algebraic subvariety of $\D_0$. It remains to show that $\cS_\Delta$ is neither empty nor the whole $\D_0$. First, let us show that $\cS_\Delta$ is not empty. As $\vert \Delta\cap\Z^2\vert \geq 4$, we can find a tropical curve $C$ of degree $\Delta$ with at least one bounded edge. Choosing a non empty subset of twisted edges inside the set of bounded edges of $C$, we can construct by tropical approximation a smooth real algebraic curves in $\LL$ with logarithmic inflection points on its real locus, see Lemma 2.11 in \cite{L2}. In particular, it belongs to $\cS_\Delta$. Now the existence of simple Harnack curves for any $\Delta$ \cite[Corollary A4]{Mikh} ensures that $\cS_\Delta$ is not the whole space.
\end{proof}

\section{The topology of $S(f)$}\label{sec:top}

In this section, we investigate the topology of the pair $(\cV(f), S(f))$. In particular, we discuss the sharpness of the bounds given in \eqref{eq:trivbd}. First, we prove Proposition \ref{prop:S(f)} describing the topology of $S(f)$ solely.

\begin{proof}[Proof of Proposition \ref{prop:S(f)}]
$S(f)$ is a possibly ramified covering of $\rp{1}$. In the unramified case, $S(f)$ is a compact covering of a circle, then a disjoint union of circles. As $\TG$ is a diffeomorphism except at the ramification points, smoothness is clear. In the ramified case, we can choose holomorphic coordinates such that, around a ramification point, $S(f)$ is the preimage of $\R$ around $0$ by the map  $ z \mapsto z^n $ for some $ n > 1 $. It implies that we can always pull-back $S(f)$ to the real projective tangent bundle over $\tilde{\cV}(f)$ inducing a compact covering of $\rp{1}$ via the tangent map of $\TG$. This pull-back is a union of smoothly embedded circles in this bundle. It projects down to a union of smoothly embedded circles intersecting transversally on $\tilde{\cV}(f)$.
\end{proof}

 We now come to the proof Theorem \ref{thm:upbd}. Recall that for a smooth projective curve of degree $d$ having transversal intersection with any of the three coordinates axes of $\cp{2}$, the degree of the its logarithmic Gau\ss{} map is $d^2$, according to Proposition \ref{deglogGauss}.

\begin{Definition}\label{def:sigma}
Let $[f] \in \LL$ and $p \in \cV(f)$ be a double point. Denote by $p_1, \, p_2 \in \tilde{\cV}(f)$ the two points projecting to $p$ and denote by $v_1, \, v_2 \in \cp{1}$ their respective image by $\TG_f$. Define $\sigma_f(p) = 1$ (respectively $-1$) if $v_1$ and $v_2$ are in the same connected component of  $\cp{1} \setminus \rp{1}$ (respectively different connected components). Define $\sigma_f(p) = 0$ whenever $v_1$ or $v_2$ sits in $\rp{1}$. When the context is clear, we will simply write $\sigma(p)$.
\end{Definition}

\begin{Proposition}\label{prop:critnod}
Let $[f] \in \D$ such that $S(f)$ is smooth and contains no double point of $\cV(f)$. Denote by $n_+$ (respectively $n_-$) the number of double points $p$ of $\cV(f)$ for which $\sigma_f(p)=1$  (respectively $\sigma_f(p)=-1$). Then, there exists a neighbourhood $\mathcal{U} \subset \LL$ of $[f]$ such that for any $[h] \in \mathcal{U}$ such that $\cV(h) $ is smooth, we have that $S(h)$ is smooth and
\[ b_0(S(h)) = b_0(S(f)) + n_- + 2 \, n_+. \]
\end{Proposition}
 
\begin{proof}
By Theorem \ref{thm:extLL}, the formula \eqref{eq:extLL} provides a continuous extension of $\ly$ to $f$. By \eqref{eq:critloc} and the assumption that $S(f)$ is smooth, it follows that the support of $\ly([f])$ is disjoint from $\rp{1}$. Note that each double point $p$ of $\cV(f)$ has to contribute either to $n_+$ or to $n_-$. By Theorem \ref{thm:extLL}, there exists a neighbourhood $\mathcal{U} \subset \LL$ such that for any $[h] \in \mathcal{U}$, the support of $\ly([h])$ is disjoint from $\rp{1}$. It implies that $S(h)$ is smooth for any $h$ such that $[h]\in \mathcal{U}$. In particular $b_0(S(h))$ is constant for any $[h]\in \mathcal{U}$ such that $\cV(h)$ is smooth. For such an $[h]$, consider a path of smooth curves converging to $[f]$. Then the connected components of $S(h)$ are of two types: the ones that converge to a connected components of $S(f)$ along this path and the ones that are contracted to a node of $\cV(f)$. To each connected component of $S(f)$ corresponds exactly one connected component of $S(h)$.  By Proposition \ref{prop:dbcov}, to each node $p \in \cV(f)$ corresponds exactly 1 or 2 connected components of $S(h)$ depending whether $p$ contributes to $n_-$ or $n_+$.
\end{proof}


We now come to the proof of Theorem \ref{thm:upbd}. The strategy is to apply Proposition \ref{prop:critnod} to a polynomial $f$ corresponding to the union of $d$ lines intersecting pairwise transversally and with prescribed signs $\sigma_f(p)$ at any double point $p\in \cV(f)$.

For $a, \, b \in \C^\ast$, define the line $\mathcal{L}_{a,b}:=\{(z,w)\in \ttor \, \big| \, az+bw=1\}$ and denote by $\gamma_{a,b}$ its logarithmic Gau\ss{} map. For two distinct lines $\mathcal{L}_{a,b}$ and $\mathcal{L}_{c,d}$ and $p:= \mathcal{L}_{a,b} \cap \mathcal{L}_{c,d}$, straightforward computations show that 
\[
\begin{array}{rl}
p =&(ad-bc)^{-1}\left( d-b,c-a\right), \\ \gamma_{a,b}(p) =&[ad-ab:bc-ab], \\ \gamma_{c,d}(p) =&[cd-bc:cd-ad].
\end{array}
\]
It follows that $\sigma(p)=0$ if and only if $ab$ or $cd$ is on the line $(ad,bc) \subset \C$. Also, $\sigma(p)=1$ if and only if $ab$ and $cd$ are on different sides of the line $(ad,bc)$. Consider now the four sets of lines
\[
\begin{array}{l}
L_1:=\{ \mathcal{L}_{e^{i\theta},e^{-i\theta}} \, \big|   \,  0<  \theta < \varepsilon \}, \; L_2:=\{ \mathcal{L}_{i \, e^{i\theta},i \, e^{-i\theta}} \, \big|   \,  0<  \theta < \varepsilon \}, \\ 
L_3:=\{ \mathcal{L}_{\rho e^{i\theta}, \rho e^{-i\theta}} \, \big|   \, M<\rho<M+1, \,  0<  \theta < \varepsilon \}, \\
L_4:=\{ \mathcal{L}_{\rho e^{i(\pi/2+\theta)}, \rho e^{i(\pi/2-\theta)}} \, \big|   \, M^2< \rho <M^2+1, \, 0<  \theta < \varepsilon \}.
\end{array}
\]
We claim that for a generic pair of lines  $(\mathcal{L},\mathcal{L}') \in L_i \times L_j$ and $p:= \mathcal{L} \cap \mathcal{L}'$, the sign $\sigma(p)$ is non zero and depends only on $(i,j)$.  Moreover, we claim that this sign is the $(i,j)$ entry of the following matrix
\[ \left(
  \begin{array}{rrrr}
    -1 & -1 &  1  & 1 \\ 
    -1 & -1 & -1 & 1 \\ 
     1  & -1 &  1 &  1 \\ 
     1  &  1  &  1  & 1 \\ 
  \end{array} \right)
\]
In order to prove this claim, it is quite helpful to notice that $\sigma(p)$ is invariant by toric translation $(z,w)\mapsto(az,aw)$. Using such translations, we can always carry $\mathcal{L}$ to $\mathcal{L}_{1,1}$ for instance. The details are left to the reader.

\begin{proof}[Proof of Theorem \ref{thm:upbd}]
For any integer $n$ satisfying $0 \leq n \leq d(d-1)/2$, we choose below a collection of $d$ lines among $L_1$, ..., $L_4$ defined by an equation $f$ such that $\sigma_f(p)=-1$ for exactly $n$ of the $d(d-1)/2$ doubles points $p$ of $\cV(f)$. Proposition \ref{prop:critnod} then provides a degree $d$ polynomial $h$ such that $\cV(h)$ and $S(h)$ are smooth and $$b_0(S(h))= d + n + 2(d(d-1)/2-n)= d + d(d-1) -n.$$
The result follows as $n$ runs from $0$ to $d(d-1)/2$.

Assume first that $n\notin\{0, d(d-1)/2\}$. Then $n=m(m-1)/2+r$ for a unique choice of integers $0\leq r<m< d$. Let $L$ be a set of $d$ generic lines such that  $\vert L \cap L_1 \vert=m-r, \; \vert L \cap L_2\vert=r , 
\vert L \cap L_3\vert=1,$ and $ \vert L \cap L_4\vert=d-m-1.$
Let now $f$ be an equation for the union of lines of $L$. The double points $p$ of $\cV(f)$ for which $\sigma_f(p)=-1$ are exactly those given as the intersection of two lines in $( L_1 \cup L_2)\cap L$ or as the intersection of a line in $L_2 \cap L$ with a line in $L_3 \cap L$. It follows that there are exactly
\[  \binom{n}{2} +r\cdot1=n\]
many such points.

Similarly, for the case $n=0$ (respectively $n=d(d-1)/2$), it suffices to consider $L\subset L_1$ (respectively $L\subset L_4$).
\end{proof}

In the last section, we show that there exist lattice polygons $\Delta$ for which the upper bound given in the inequality \eqref{eq:trivbd} is not sharp. We now discuss the sharpness of the lower bound. The proposition below implies that the lower bound is not sharp whenever $g_\Delta$ is odd. Recall by \cite{Kho} that the generic curve $\cV(f)$ of $\LL$ has genus $g_\Delta$.

\begin{Proposition}
Let $[f] \in \D_0$ be any class of Laurent polynomial such that $S(f)$ is smooth and connected. Then,  $g_\Delta$ has to be even.
\end{Proposition}

\begin{proof}
Since $S(f)$ is connected by assumption, the description $S(f) = \TG_f^{-1} (\rp{1})$ implies that $S(f)$ is a single loop separating $\cV(f)$ into two connected components. Each 
such connected component $\mathcal{C}$ is a genus $g' \leq g$ subsurface of $\cV(f)$ with one boundary component. It follows that $g'$ is determined by $\chi(\mathcal{C})$. We show that $\chi (\mathcal{C})$ itself is determined by the map $\A$ restricted to $S(f)$ only. This fact implies that both components $\mathcal{C}$ have the same genus $g'$, implying the result. 

Since there is no ramification point on $S(f)$ by assumption, $\TG_f$ is monotone on $S(f)$, so that 
the curvature of $\A\big(S(f)\big)$ does not change as long as $\A$ is an immersion. By \cite[Lemma 1]{Mikh}, the local degree of 
$\A_{\vert_\mathcal{C}}$ in $\R^2 \setminus \A \big( S(f) \big)$ increases by one when crossing $\A \big(S(f)\big)$ from the convex to the 
concave side. Hence, the number of preimages of any connected component of $\R^2 \setminus \A \big( S(f) \big)$ in the amoeba plane is determined by $\A \big(S(f)\big)$. We can then compute $\chi(\mathcal{C})$ by adding the Euler characteristics of the 2-cells of $\R^2 \setminus \A \big( S(f) \big)$, counted with multiplicity. It follows that $\chi ( \mathcal{C})$ only depends on $\A\big(S(f)\big)$.
\end{proof}

The first interesting case that is not covered by the above proposition is $g_\Delta=2$. If we want to construct a curve 
$\cV(f)$ of genus $2$ whose critical locus $S(f)$ is connected, then $S(f)$ separates $\cV(f)$ into two homeomorphic halves. Let 
$\mathcal{H}$ be one of them. The map $\A$ restricted to $\mathcal{H}$ has to deform to the immersion of a bouquet of two circles, which is the skeleton of $\mathcal{H}$. On the one hand, an immersed bouquet of two circles has to bound at least 3 compact connected components. On the other hand, Proposition 2.5 in \cite{FPT} implies that the number of compact connected components of $\A \big( \cV(f) \big)$ is bounded from above by $g_\Delta$. 

In the more general context of harmonic amoebas (see \cite{L}), we can construct a Riemann surface $\mathcal{V}$ in $\mathcal{M}_{2,3}$ together with an harmonic amoeba map with connected critical locus $S$, see figure \ref{pic:haram}.

\begin{figure}[h]
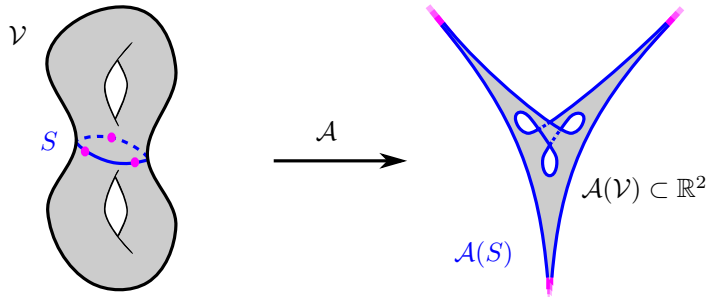
\caption{A harmonic amoeba of a Riemann surface $\mathcal{V}  \in \mathcal{M}_{2,3}$ with connected critical locus $S$.}
\label{pic:haram}
\end{figure}

The point that we want to make here is that the injectivity of the order map defined in \cite{FPT} might provide non trivial obstructions to low values of $b_0\big(S(f)\big)$, in the sense Conjecture \ref{conj}.

\section{Examples}\label{sec:ex}

In the following, we illustrate our results with some computations. The simplest interesting case is for $\vert \Delta \cap \Z^2\vert=4$. Note that the action of the torus $\ttor$ on itself induces homogeneities on $\LL$. More precisely, the toric translations of the coordinates $(z,w) \mapsto (\alpha z, \beta w)$ induce isomorphisms on curves $\cV(f)$ and do not affect $\TG_f$, see \cite[Chapter 9, Section B]{GKZ}. 

\subsection{Bi-degree $(1,1)$ curves in $\cp{1} \times \cp{1}$}

We can choose $\Delta$ to be the convex hull of the set $A=\left\lbrace 1,z,w,zw \right\rbrace$. After reducing homogeneities, generic polynomials have the form $ f(z,w)=1+z+w+\alpha zw$. We can parametrize $\cV(f)$ by looking at the pencil of lines passing through a chosen point of $\cV(f)$. For example, considering the pencil of lines $\left[u:v\right] \mapsto \left\lbrace (-1 + uz, vz) \, \vert \, z \in \C \right\rbrace $ through $(-1,0)$ and expressing the solutions of $f(-1 + uz, vz)=0$ as a function of $t:=u/v$ provides the parametrization 
\[ t \mapsto \left( -\dfrac{1+t}{\alpha t}, -\dfrac{1+t(1-\alpha)}{\alpha} \right) \]
of $\cV(f)$.
Its logarithmic Gau\ss{} map is given by $$ \gamma_f(t)= \Big[ \; t (1+t)(1-\alpha) : 
1+t(1-\alpha) \; \Big].$$
By differentiating the ratio of the two coordinates of $\gamma_f$, we deduce that the ramification points of $\gamma_f$ are given by the parameters $t$ satisfying
\[ t^2 (1-\alpha) +2t +1=0\]
namely $t_\pm= (1 \pm \sqrt{\alpha})/(-1+\alpha)$. After computing $\gamma_f(t_\pm)= -2t_\pm-1$, it follows that 
$$ \alpha \in \cS_\Delta \Leftrightarrow 2t_++1 \in \R \text{ or } 2t_-+1 \in \R \Leftrightarrow t_+ \in \R \text{ or } t_- \in \R \Leftrightarrow  \alpha \in \R_{\geq 0}.$$
Amoebas of curves corresponding to $\alpha>0$ possess a pinching point, see \cite{Mikh}. In that case, $S(f)$ consists of two circles intersecting on the sphere $\cV(f)$ such that one of them is contracted to the pinching point by $\A$.

\begin{figure}[h]
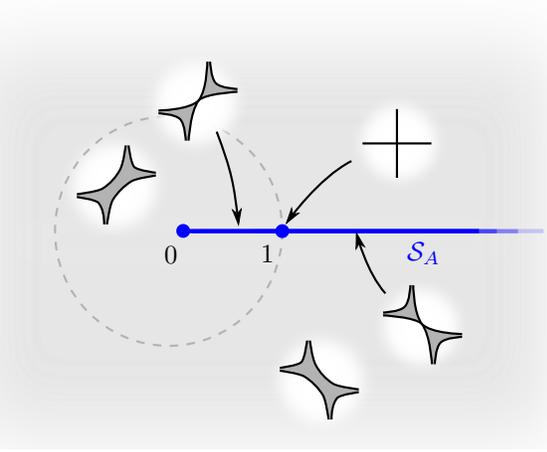
\caption{$\cS_\Delta$ inside the $\alpha$-space, and corresponding amoebas.}
\label{pic:cfsp11}
\end{figure}

\subsection{The case $ \Delta =\conv (\left\lbrace 1,z,w,z^2 \right\rbrace )$}
Moding out the action of $\ttor$, $f$ is generically of the form $ f(z,w)=1+z+w+\alpha z^2$. Similar computations show that the complement of $\cS_\Delta$ is connected.

\subsection{The case $ \Delta =\conv (\left\lbrace 1,z,w,zw,z^2 \right\rbrace) $}
After reduction, we can consider polynomials of the form $f(z,w)=1+z+w+\alpha zw+\beta z^2$. Following Remark \ref{rem:ll}, we compute the resultant $R:= R_\Delta \left( f, \gamma_{f,[u:v]}, g \right)$ using Singular \cite{Sing}, where $g:= \det \big( grad \, f, grad \, \gamma_{f,[u:v]} \big).$ We obtain that 
\[
\begin{array}{rl}
 R(f)= &\alpha^3 \beta v^2 (u+v)^2 (-\alpha+\alpha^2+\beta) \big[ (4\beta-1)(\beta+\alpha^2-\alpha)u^4 \\
 & \\
 & -2(4\beta-1)(-3\beta+\alpha+\alpha^2)u^3v + (-\alpha^2+\alpha-13\beta+48\beta^2)v^2u^2 \\
 & \\
  & +4\beta(\alpha-3+8\beta+2\alpha^2)uv^3-4\beta(\alpha-1)^2v^4 \big]
\end{array}
\]
Note first that the factor $(-\alpha+\alpha^2+\beta) $ of $R$ corresponds to the set of polynomials $f$ for which $\cV(f)$ is singular.
The factor $v^2 (u+v)^2$ of $R$ does not corresponds to branching points of the logarithmic Gau\ss{} map. Otherwise, it would implies that $\D_0 \subset \cS_\Delta$ in contradiction with Theorem \ref{thm:sa}. By Proposition \ref{deglogGauss} and the Riemann-Hurwitz formula \eqref{eq:RHformula}, the branching divisor $B_{\gamma_f}$ has degree 4 for generic $f$. It follows that 
\[ \ly(f) = R(f)\cdot \big( \alpha^3 \beta v^2 (u+v)^2 (-\alpha+\alpha^2+\beta) \big)^{-1}. \]
Here, few remarks are worth to be made concerning the possibility to extend $\ly(f)$. First, the coefficients of $\ly(f)$ never vanish simultaneously, except for $\alpha=\beta=0$. It implies that $\ly(f)$ is well defined as long as $f$ is not the equation of a line. Secondly, when $\cV(f)$ is singular, or equivalently when $\beta=\alpha-\alpha^2$, we have that $f(z,w)=(\alpha z +1)\big(1+(1-\alpha)z+w\big)$ implying that $\cV(f)$ is the union of a generic line (for $\alpha\neq1$) with the binomial cylinder $\{z=-1/\alpha\}$. On the one hand, computation shows that
$$\ly(f)= 4\alpha v\big((2\alpha-1) u+(\alpha-1) v)^3.$$
On the other hand, if we denote by $p \in \tilde{\cV}(f)$ (respectively $q$) the preimage of the singular point of $\cV(f)$ sitting on the line (respectively on the binomial cylinder), then $$\ly(f)= \TG_f(q) + 3\cdot \TG_f(p)$$
as a divisor. This corresponds to the extension of the $\ly$ map  of Theorem \ref{thm:extLL}.

We now show that the complement of  $\cS_\Delta$ is disconnected by exhibiting two polynomials $f$ whose corresponding critical locii $S(f)$ have differently many connected components. To that aim, we restrict ourselves to the subspace $\beta=-\alpha$ with corresponding polynomial $f_\alpha$ and compute 
\[ \ly(f_\alpha)=(-4\alpha^2+7\alpha+2)u^4+(34\alpha+8\alpha^2+8)u^3v \hspace{0.5cm} \]
\[ \hspace{2.5cm} +(47\alpha+14)v^2u^2+(-8\alpha^2+12+28\alpha)uv^3+ (4\alpha^2+4-8\alpha)v^4. \]
As the dependence of $\ly(f_\alpha)$ in $\alpha$ is quadratic, the solutions of $\ly(f_\alpha)=0$ can be given as two functions $\alpha_\pm([u:v])$. Considering $[u:v] \in \rp{1}$, the functions $\alpha_\pm([u:v])$ provide a parametrization of $\cS_\Delta$ restricted to the space $\beta=-\alpha$ which turns out to be disconnected, see Figure \ref{fig:cfsp2}. 

\begin{figure}[h]
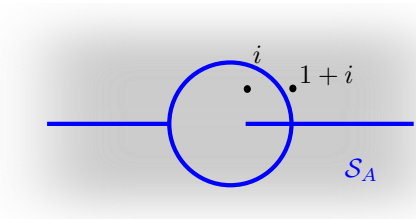
\caption{$\cS_\Delta$ inside the $\alpha$-space.}
\label{fig:cfsp2}
\end{figure}

Consider for instance the polynomials $f_i$ and $f_{1+i}$ sitting in different connected components. Using the parametrization 
\[t\mapsto \Big( \frac{-1-t+\alpha}{\alpha t - \alpha}, \frac{-(1+t)t+\alpha}{\alpha t - \alpha} \Big)\]
of $\cV(f_\alpha)$ and the corresponding logarithmic Gau\ss{} map
\[\gamma_{f_\alpha}(t)=\big[ (\alpha-2t-1+t^2)(1+t-\alpha):(\alpha-2)(\alpha-t-t^2)\big], \]
 we can show that the critical locii $S(f_i)$ and $S(f_{1+i})$ have respectively 1 and 2 connected components, see Figure \ref{fig:examples}.
 
 \begin{figure}[h]
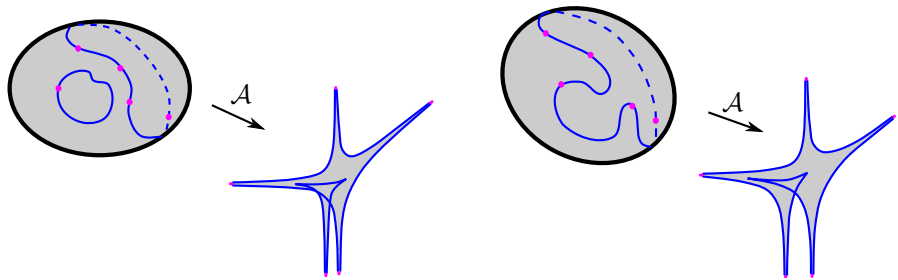
\caption{The amoeba map on $\cV(f_{1+i})$ (left) and $\cV(f_i)$ (right): the singular locii $S(f_{1+i})$ and $S(f_i)$ are drawn in blue, as well as their image by $\A$. The intersection points of $\cV(f_{1+i})$ and $\cV(f_i)$ with $\Xd^\infty$ are drawn in purple.}
\label{fig:examples}
\end{figure}

\bibliographystyle{alpha}
\bibliography{Draft}

\vspace{1cm}
\noindent
\textit{Email address}: lang.lionel@math.uu.se

\end{document}